\documentclass[11pt]{article}
\usepackage{graphicx} 
\usepackage{hyperref}
\usepackage{amsmath,amssymb,amsfonts,amsthm,tikz,xspace,fullpage}
\usepackage[capitalise, noabbrev]{cleveref}
\usetikzlibrary{calc}
\usepackage{comment}

\title{The radius capture number}

\newtheorem{theorem}{Theorem}

\newtheorem{corollary}[theorem]{Corollary}
\newtheorem{lemma}[theorem]{Lemma}

\newtheorem{problem}[theorem]{Problem}
\newtheorem{question}[theorem]{Question}
\newtheorem{remark}[theorem]{Remark}

\newtheorem{proposition}[theorem]{Proposition}

\theoremstyle{definition}

\DeclareMathOperator{\rad}{rad}
\DeclareMathOperator{\diam}{diam}
\DeclareMathOperator{\ecc}{ecc}
\DeclareMathOperator{\BoxProduct}{\mathbin{\Box}}

\DeclareMathOperator{\rc}{rc}
\newcommand{\cvSpan}[1]{\sigma^{\BoxProduct}_V(#1)}

\begin{document}

\author{
Tanja Dravec $^{1,2,}$\thanks{Email: \texttt{tanja.dravec@um.si}}
\and
Vesna Ir\v{s}i\v{c} Chenoweth$^{2,3,}$\thanks{Email: \texttt{vesna.irsic@fmf.uni-lj.si}}
\and
Andrej Taranenko $^{1,2,}$\thanks{Email: \texttt{andrej.taranenko@um.si}}
}

\maketitle

\begin{center}
$^1$ Faculty of Natural Sciences and Mathematics,  University of Maribor, Slovenia\\
\medskip

$^2$ Institute of Mathematics, Physics and Mechanics, Ljubljana, Slovenia\\
\medskip

$^3$ Faculty of Mathematics and Physics,  University of Ljubljana, Slovenia\\
\medskip
\end{center}

\begin{abstract}
   In the classic cop and robber game, two players--the cop and the robber--take turns moving to a neighboring vertex or staying at their current position. The cop aims to capture the robber, while the robber tries to evade capture. A graph $G$ is called a cop-win graph if the cop can always capture the robber in a finite number of moves. In the cop and robber game with radius of capture $k$, the cop wins if he can come within distance $k$ of the robber. The radius capture number ${\rm rc}(G)$ of a graph $G$ is the smallest $k$ for which the cop has a winning strategy in this variant of the game.

   In this paper, we establish that ${\rm rc}(H) \leq {\rm rc}(G)$ for any retract $H$ of $G$. We derive sharp upper and lower bounds for the radius capture number in terms of the graph’s radius and girth, respectively. Additionally, we investigate the radius capture number in vertex-transitive graphs and identify several families $\cal{F}$ of vertex-transitive graphs with ${\rm rc}(G)=\rad(G)-1$ for any $G \in \cal{F}$.
We further study the radius capture number in outerplanar graphs, Sierpi\' nski graphs, harmonic even graphs, and graph products. Specifically, we show that for any outerplanar graph $G$, ${\rm rc}(G)$ depends on the size of its largest inner face. For harmonic even graphs and Sierpi\' nski graphs $S(n,3)$, we prove that ${\rm rc}(G)=\rad(G)-1$. Regarding graph products, we determine exact values of the radius capture number for strong and lexicographic products, showing that they depend on the radius capture numbers of their factors. Lastly, we establish both lower and upper bounds for the radius capture number of the Cartesian product of two graphs.   
\end{abstract}


\noindent
{\bf AMS Subj.\ Class.\ (2020):} 05C12, 05C57

\maketitle

\section{Introduction}

The cop and robber game, introduced by Nowakowski and Winkler~\cite{NW-1983} and Quilliot~\cite{Qu-1983}, was intensively studied by numerous authors~\cite{BKP-2012, G-2010}. The cop and robber game is a pursuit-evasion game played on the vertices of finite undirected graphs. First, the cop occupies his initial position, then the robber occupies another vertex. Then they (starting with the cop) alternately either move to an adjacent vertex or stay in the same vertex. Both players know the position of the opposing player. The aim of the cop is to capture the robber, i.e. to occupy the vertex that is occupied by the robber, while the robber wants to evade the cop. A graph $G$ is a {\em cop-win} graph if the cop has a strategy to capture the robber after a finite number of moves regardless of the robber's playing strategy. 

A {\em corner} of a graph $G$ is a vertex $u$ such that there exists $v \in V(G)$ with $N[u] \subseteq N[v]$. A graph $G$ is {\em dismentlable} if there exists a sequence of corners whose iterative deletion results in $K_1$. Cop-win graphs are exactly dismantlable graphs~\cite{NW-1983, Qu-1983}. Since the introduction many variants of the cop and robber game were developed. The most natural one replaces a single cop with a set of cops. In this sense the {\em cop number} of a graph $G$, $c(G)$, was introduced by Aigner and Fromme~\cite{AF-1984}, as the minimum number of cops needed to capture the robber in $G$. The distance $k$ cop and robber game was introduced by Bonato and Chiniforooshan~\cite{BC-2009}. In this game the cop wins if he is at distance at most $k$ from the robber. The distance $k$ cop number of a graph $G$, $c_k(G)$, is the minimum number of cops necessary to win at a given distance $k$. The game was further studied in~\cite{CCNV-2011} under a different name: {\em the cop and robber game with radius of capture $k$}. We will use this name and denote by ${\cal{CWRC}}(k)$ the set of graphs in which cop has a winning strategy in the cop and robber game with radius of capture $k$ regardless of how robber is playing. In~\cite{BCP-2010} an $O(n^{2s+3})$ algorithm for testing weather $c_k(G) \leq s$ is presented. Moreover, upper and lower bounds in terms of order are proved for $c_k(G)$. Chalopin et al. considered the cop and robber game with radius of capture $1$ on bipartite graphs~\cite{CCNV-2011}. They characterized bipartite graphs in ${\cal{CWRC}}(1)$ and explain that the characterization for general graphs seems to be quite challenging. Graphs in ${\cal{CWRC}}(1)$ were studied also in~\cite{Dang-2011}, where the cop and robber game with radius of capture 1 was considered on square-grids, $k$-chordal graphs and outerplanar graphs.

In this paper we define the \emph{radius capture number} of a graph $G$, $\rc(G)$, as the smallest $k$ such that $G \in {\cal{CWRC}}(k)$, i.e. $\rc(G)=\min{\{k \mid G \in {\cal{CWRC}}(k)\}}$. Note, that we always consider the game with exactly one cop. Hence, if $G$ is a cop-win graph, then $\rc(G)=0$. If $G$ is not connected, then $G \notin {\cal{CWRC}}(k)$ for any $k\in {\mathbb{N}}_0$, i.e.\ robber has a winning strategy in the cop and robber game with radius of capture $k$. Namely, if cop in his first move selects a vertex in the component $G_1$ of $G$, then robber selects a vertex in any component different from $G_1$ and wins the game. For this reason if not explicitly stated we will always consider connected graphs. 

In Section~\ref{sec:prelim} we start with basic definitions and present some preliminary results. We show that if $G\in {\cal{CWRC}}(k)$, then ${H \in \cal{CWRC}}(k)$ holds for any retract $H$ of $G$. In Section~\ref{sec:bounds} we present tight upper and lower bounds for the radius capture number  of an arbitrary graph. Section~\ref{sec:vertexTransitive} is devoted to vertex-transitive graphs. We consider different families of vertex-transitive graphs and determine exact values of the radius capture number  for all graphs in presented families. We prove that any generously-transitive graph $G$ satisfies $\rc(G)=\rad(G)-1$. Moreover, we show that distance-transitive graphs and generalized Johnson graphs are generously-transitive and thus we obtain the exact value of the radius capture number  for any distance-transitive and any generalized Johnson graph. We also show that not all vertex-transitive graphs $G$ satisfy $\rc(G)=\rad(G)-1$. Finally, in Section~\ref{sec:other} we investigate the radius capture number  of outerplanar graphs, harmonic even graphs, Sierpiński graphs and graph products. We prove that for any outerplanar graph $G$, $\rc(G)$ depends on the size of the largest inner face. We present both upper and lower bounds for the radius capture number  of the Cartesian product of two graphs. Moreover, for the strong and lexicographic product we get exact values of the radius capture number  which depend on the radius capture number  of the factors.

\section{Preliminaries}\label{sec:prelim}

Let $G=(V,E)$ be a graph. The {\em distance} between two vertices $u,v \in V(G)$, $d_G(u,v)$, is the length of a shortest $u,v$-path in $G$. For $C \subseteq V(G)$ and $v\in V(G)$, $d_G(v,C)=\min{\{d_G(v,x)\mid x \in V(C)\}}$. The {\em open neighborhood} of $u \in V(G)$, $N_G(u)$, is the set of neighbors of $u$ and the {\em closed neighborhood} of $u$ is the set $N_G[u]=N_G(u) \cup \{u\}$. For $S \subseteq V(G)$ the {\em closed neighborhood of $S$} is the set $N_G[S]=\bigcup_{u \in S}N_G[u]$. When the graph $G$ is clear from the context, we shall omit it in the notation. The subgraph of $G$ induced by $S$ is denoted by $G[S]$ and for $v \in V(G)$ we write $G-\{v\}$ instead of $G[V(G) \setminus \{v\}].$ For a connected graph $G$ and $v \in V(G)$ the {\em eccentricity} of $v$, $\ecc(v)$, is the maximum distance between $v$ and any other vertex $u$ of G, i.e.\ $\ecc(v)=\max{\{d(v,u)\mid u \in V(G)\}}$. In a graph $G$, an \emph{antipode} of a given vertex $v$ is a vertex $u$ with $d(v,u)=\ecc(v)$. The {\em diameter} of a graph $G$ is defined as diam$(G)=\max{\{\ecc(v)\mid v\in V(G)\}}$ and the {\em radius} of $G$ is rad$(G)=\min{\{\ecc(v) \mid v \in V(G)\}}$. The {\em girth} of a graph $G$, $g(G)$, is the length of one of its (if any) shortest graph cycles. In this paper acyclic graphs are considered to have girth $0$.  

A {\em homomorphism} $f$ from a graph $G$ to a graph $H$ is a function $f:V(G) \to V(H)$ that maps edges to edges, i.e.\ if $xy \in E(G)$ then $f(x)f(y)\in E(H)$. An induced subgraph $H$ of $G$ is a {\em retract} of $G$ if there exists a homomorphism $f: V(G) \to V(H)$ such that $f(h)=h$ holds for any $h \in V(H)$. The map $f$ is in this case called a {\em retraction}. It was proved in~\cite{BI-1993} that the cop number of a retract $H$ of $G$ never exceeds the cop number of $G$. We will use the same ideas to prove that if $H$ is a retract of $G$ then $G \in {\cal{CWRC}}(k)$ implies that $H \in {\cal{CWRC}}(k)$. We will start by presenting some well-known properties of retracts. 

\begin{remark}\label{rem: retract}
    Let $H$ be an induced subgraph of a connected graph $G$ and $f:V(G) \to V(H)$ a retraction. Then $d_G(x,y) \geq d_H(f(x),f(y))$ for any $x,y \in V(G)$. 
\end{remark}

 Let $H$ be a retract of $G$. Remark~\ref{rem: retract} implies that for any $x,y \in V(H) \subseteq V(G)$, $d_G(x,y)\geq d_H(f(x),f(y))=d_H(x,y)$. Hence $d_G(x,y)=d_H(x,y)$ holds for all $x,y \in V(H)$ and thus a retract $H$ of a graph $G$ is always isometric subgraph of $G$. 

\begin{theorem}\label{thm: retracts}
Let $H$ be a retract of a connected graph $G$ and $k$ an arbitrary non-negative integer. If $G \in {\cal{CWRC}}(k)$, then $H \in {\cal{CWRC}}(k)$.
\end{theorem}

\begin{proof}
Let $f:V(G) \to V(H)$ be a retraction and let $G \in {\cal{CWRC}}(k)$. Hence the cop has a winning strategy on $G$ for the cop and robber game with radius of capture $k$. Consider two parallel cop and robber games with radius of capture $k$, one played on $G$ and the other played on $H$. Since $H$ is an induced subgraph of $G$, the game in $H$ can be considered as being played on $G$. The cop on $G$ plays by using his optimal strategy (he has a winning strategy on $G$ since $G \in  {\cal{CWRC}}(k)$) and uses the following shadow strategy for playing on $H$. If in his $i^{\text{th}}$ move the cop is in $v$ when playing on $G$, then he will move to $f(v)$ when playing on $H$ and $f(v)$ is then called the shadow of $v$. If on $G$ cop moves from $u$ to $v$, then on $H$ he moves from $f(u)$ to $f(v)$. Hence $u=v$ or $uv \in E(G)$ and since $f$ is homomorphism, $f(u)=f(v)$ or $f(u)f(v)\in E(H)$. Thus the move from $f(u)$ to $f(v)$ is a valid move for the cop on $H$. It remains to prove that the described cop's strategy is a winning strategy for the cop and robber game with radius of capture $k$ played on $H$.

Suppose that cop in his $i^{\text{th}}$ move, when playing on $G$, selects the vertex $c_i$. Then, by the above described strategy, he occupies $c_i'=f(c_i)$ in the game on $H$. Moreover, let $r_i'$ be the vertex occupied by the robber in his $i^{\textrm{th}}$ move when he plays on $H$. Then the cop imagines that this was robbers move in $G$ and follows his strategy on $G$. Since $f$ is a homomorphism, Remark~\ref{rem: retract} implies that $d_H(f(r_i'),f(c_i))=d_H(r_i',c_i') \leq d_G(r_i',c_i)$. Since $G \in {\cal{CWRC}}(k)$, there exists $i$ such that $d_G(r_i',c_i) \leq k$ and thus $d_H(r_i',c_i') \leq k$ which implies that $H \in {\cal{CWRC}}(k)$.
\end{proof}

In~\cite{CCNV-2011} authors characterized bipartite graphs that are in ${\cal{CWRC}}(1)$. The proof of that result could be simplified by using Theorem~\ref{thm: retracts}. If $G$ is a bipartite graph from ${\cal{CWRC}}(1)$, then the cop has a winning strategy in the cop and robber game with radius of capture $1$. Let $v$ be the next-to-last position of the robber in the game played on $G$. The proof in~\cite{CCNV-2011} goes by induction and the main step of the proof is to see that $G'=G-\{v\}$ belongs to ${\cal{CWRC}}(1)$. Authors also note that $G'$ is retract of $G$ and thus by Theorem~\ref{thm: retracts} it directly follows that $G' \in {\cal{CWRC}}(1)$.

\section{Bounds and complexity}\label{sec:bounds}

In this section we prove a sharp upper and a lower bound for the radius capture number  of an arbitrary connected graph $G$.

\begin{theorem}\label{thm:boundRad}
    If $G$ is a connected graph, then $\rc(G) \leq \rad(G)-1$.
\end{theorem}
\begin{proof}
    Let $G$ be a connected graph and let $k=\rad(G)$. We will prove that $G \in {\cal{CWRC}}(k-1)$. Let $v \in V(G)$ be a vertex with the smallest eccentricity, i.e. $\ecc(v)=\rad(G)$. The cop's strategy is that he occupies $v$ in his first move. Let $u$ be the vertex occupied by the robber in his first move and let $v'$ be the neighbor of $v$ on a shortest $u,v$-path. The cop then moves from $v$ to $v'$. Since $d_G(v',u) \leq k-1$ the cop wins the cop and robber game with radius of capture $k-1$ and thus $G \in {\cal{CWRC}}(k-1)$. 
\end{proof}

The bound in Theorem~\ref{thm:boundRad} is sharp. For example, if $n$ is even, then $\rc(C_n)=\frac{n}{2}-1=\rad(C_n)-1$. There are also many other graphs that attain this bound. We will prove in the next section that Johnson and Kneser graphs attain this bound. 

If for any two vertices $x,y \in V(G)$ with $d(x,y)=\rad{(G)}$ it holds that for any $x' \in N[x]$ there exists $y' \in N[y]$ with $d(x',y')=\rad{(G)}$, then $\rc(G)=\rad(G)-1$. Thus we can easily get a sufficient condition for graphs $G$ with $rc(G)=\rad{(G)}-1$. But it is not clear if this condition is also necessary. 

\begin{problem}\label{prob:1}
Characterize connected graphs $G$ with $\rc(G)=\rad(G)-1$.
\end{problem}

We mention the following corollary of the upper bound from Theorem \ref{thm:boundRad} and an algorithm given in \cite{BCP-2010}.

\begin{corollary}
    \label{thm:poly}
    Determining $\rc(G)$ for a graph $G$ on $n$ vertices can be done in time $O(n^6)$.
\end{corollary}

\begin{proof}
    By definition, the smallest such $r$ for which $c_k(G) \leq 1$ equals $\rc(G)$. Using the polynomial algorithm from \cite{BCP-2010} which checks if $c_k(G) \leq 1$ in time $O(n^5)$, combined with the upper bound from Theorem \ref{thm:boundRad} the following simple algorithm determines $\rc(G)$ in time $O(n^6)$. For $k = 0, 1, \ldots, \rad(G)-1$, check if $c_k(G) \leq 1$. If yes, return $k$; otherwise proceed to the next $k$. Thus the algorithm from~\cite{BCP-2010} implies that determining $\rc(G)$ can be done in polynomial time as asserted.
\end{proof}

\begin{theorem}\label{thm:BoundGirth}
    If $G$ is a connected graph, then $\rc(G) \geq \left\lfloor \frac{g(G)}{2}\right\rfloor-1$.
\end{theorem}
\begin{proof}
    Let $C:x_1,x_2,\ldots, x_k,x_1$ be a shortest cycle of $G$ (of length $k$) and let $c_1=v$ be the vertex occupied by the cop in his first move. The strategy of robber is to first select a vertex $r_1$ on $C$ which is at distance at least $\left\lfloor\frac{k}{2}\right\rfloor$ from $c_1$ and always only move on $C$ in such a way that after each of the robber's moves it holds true that $d_G(c_i,r_i) \geq \left\lfloor\frac{k}{2}\right\rfloor$, where $c_i$ and $r_i$ denote the vertices occupied by the cop and the robber in the $i^{\text{th}}$ move, respectively. 
    
    In the rest of the proof we show that for any $i$, such a vertex $r_i$ for the robber exists. The proof goes by induction on $i$. Since $k=g(G)$, there exists $j \in [k]$ such that $d_G(c_1,x_j) \geq \left\lfloor \frac{k}{2}\right\rfloor$ and hence $r_1=x_j$ has the desired property. Now, let $c_i$ be the vertex occupied by the cop in his $i^{\text{th}}$ move and let $r_{i-1}=x_{j'}$ for some $j' \in [k]$ (remember, that the robber's strategy is to only move on $C$). By the induction assumption $d_G(c_{i-1},r_{i-1}) \geq \lfloor \frac{k}{2} \rfloor$. If $d_G(c_i,r_{i-1}) \geq \lfloor \frac{k}{2} \rfloor$, then let $r_i=r_{i-1}$. Otherwise, $d_G(c_i,r_{i-1})=\lfloor \frac{k}{2} \rfloor -1$. If $c_i$ is on the cycle $C$ there clearly exists $x_{j''}\in N(x_{j'})\cap V(C)$ such that $d_G(c_i,x_{j''})=\lfloor \frac{k}{2} \rfloor$, thus $r_{i}=x_{j''}$ and this case is done. If $c_i$ is not a vertex of the cycle $C$ and $x_{j''}\in N(x_{j'})\cap V(C)$ such that $d_G(c_i,x_{j''})=\lfloor \frac{k}{2} \rfloor -1$, then the vertices $c_i, x_{j'}$ and $x_{j''}$ lie on an induced cycle of length at most $\lfloor \frac{k}{2} \rfloor -1 + \lfloor \frac{k}{2} \rfloor -1 + 1 = 2\lfloor \frac{k}{2} \rfloor -1 \leq k-1$, a contradiction with the fact that $k$ is the length of a shortest cycle. Thus, the robber can move to any of the two neighbors of $x_{j'}$ on the cycle $C$. This completes the proof.
\end{proof}

The bound from Theorem~\ref{thm:BoundGirth} is sharp. For any $n \geq 4$, $\rc(C_n)=\lfloor \frac{n}{2} \rfloor -1 = \lfloor \frac{g(G)}{2} \rfloor-1$. We will prove in Section~\ref{sec:outerplanar} that outerplanar graphs in which all inner faces are of the same length also attain this bound.

\section{Vertex-transitive graphs}\label{sec:vertexTransitive}
In this section we determine the value for the radius capture number  for different classes of vertex-transitive graphs. All presented families of graphs in this section achieve the upper bound from Theorem~\ref{thm:boundRad}. However, this is not true for all vertex-transitive graphs, as we show towards the end of this section.

A graph $G$ is \emph{vertex-transitive} if for any $u,v \in V(G)$ there exists an automorphism that maps $u$ to $v$. A graph $G$ is \emph{generously-transitive} if for any $u,v \in V(G)$ there exists an automorphism $\varphi$ with $\varphi(v)=u$ and $\varphi(u)=v$. In the next result we prove that all generously-transitive graphs achieve the bound from Theorem~\ref{thm:boundRad}.

\begin{theorem}\label{thm:strongAut}
    If $G$ is generously-transitive, then $\rc(G)=\rad(G)-1$.
\end{theorem}
\begin{proof}
Let $k=\rad(G)$. By Theorem~\ref{thm:boundRad}, $\rc(G) \leq k-1$. To prove the reversed inequality we need a strategy for the robber that guarantees that at each time of the game the distance between the cop and the robber is at least $k-1$. Let $c_1$ be the vertex chosen by the cop in his first move. Since $k=\rad(G)$, there exists a vertex $x$ with $d(c_1,x)\geq k$. Let $r_1=x$ be the vertex chosen by the robber in his first move. Now, we will prove inductively that after each robber's move it holds true that $d(c_i,r_i)\geq k$, where $c_i,r_i$ denote the vertices occupied by the cop and the robber in their $i^{\text{th}}$-move, respectively. Suppose that $d(c_i, r_i)\geq k$ and let $c_{i+1}$ be the vertex occupied by the cop in his next move. If $d(c_{i+1},r_i) \geq k$, then robber has a strategy to stay in $r_i$, i.e.\ $r_{i+1}=r_i$. Otherwise, it holds that $d(c_{i+1},r_i)=k-1$. By the assumption, there exists an automorphism $\varphi: V(G) \to V(G)$ with the property $\varphi(c_{i+1})=r_i$ and $\varphi(r_i)=c_{i+1}$. Since $c_{i+1}$ has a neighbor $c_i$ that is at distance $k$ from $r_i$, it holds that $f(c_{i+1})=r_i$ has a neighbor $x=f(c_i)$ that is at distance $k$ from $f(r_i)=c_{i+1}$. Thus let $r_{i+1}=x$, which implies that $d(c_{i+1},r_{i+1}) =k$. Therefore after each robber's move the distance between both players is at least $k$ and hence $\rc(G) \geq k-1$.  
\end{proof}

Cayley graphs of Abelian groups have generously transitive automorphism groups~\cite{GR-2001}. Thus we get the following.

\begin{corollary}
    \label{cor:cayley}
    If $G$ is a Cayley graph of an Abelian group, 
    then $\rc(G) = \rad(G) - 1$.
\end{corollary}

A graph $G$ is \emph{distance-transitive} if for any vertices $u,v,x,y \in V(G)$ for which $d(u,v)=d(x,y)$, there is an automorphism of $G$ that maps $u$ to $x$ and $v$ to $y$~\cite{BS-1971}.

\begin{lemma}\label{l:Distance-transitive}
    If $G$ is a distance-transitive graph, then it is generously-transitive.
\end{lemma}

\begin{proof}
    Let $G$ be a distance-transitive graph and let $u,v\in V(G)$ be arbitrary. Denote by $i=d(u,v)$. Since $G$ is distance-transitive, for any two pairs of vertices at distance $i$, also for $(u,v)$ and $(v,u)$, there exist an automorphism that maps $u$ to $v$ and $v$ to $u$.   
\end{proof}

Theorem~\ref{thm:strongAut} and Lemma~\ref{l:Distance-transitive} imply the following result.

\begin{corollary}\label{thm:distance-transitive}
    If $G$ is a distance-transitive graph, then $\rc(G)=\rad(G)-1$. 
\end{corollary}

Let $n > k > i$ be non-negative integers. The generalized Johnson graph, $J(n,k,i)$ is the graph whose vertices are the $k$-subsets of an $n$-set, where vertices $A$ and $B$ are adjacent whenever $|A \cap B|=i$. We will consider two special cases of generalized Johnson graphs,  the Johnson graphs and Kneser graphs. Johnson graphs are generalized Johnson graphs with $i=k-1$, i.e.\ two $k$-subsets $A,B$ of an $n$-set $X$ are adjacent if and only if $|A \cap B|=k-1$. Johnson graphs are denoted by $J(n,k)=J(n,k,k-1)$ and satisfy $\rad(J(n,k))=k$~\cite{HS-1993}. Moreover,
graph $J(n,k)$ is isomorphic to $J(n,n-k)$ and hence we may always assume that $k \leq \frac{n}{2}$. Since Johnson graphs are distance-transitive~\cite{HO-2007}, we get the following.

\begin{corollary}\label{thm:Johnson}
    Let $n$ be a positive integer and $1 \leq k \leq \frac{n}{2}$. Then $\rc(J(n,k))=k-1 =\rad(J(n,k))-1.$
\end{corollary}

There are many other well-known vertex-transitive graph classes that are contained in the class of distance-transitive graphs, such as Grassmann graphs, odd graphs, hypercubes, Hamming graphs, triangular graphs, complete $k$-partite graphs, etc. Corollary~\ref{thm:distance-transitive} implies that for any such graph $G$ it holds that $\rc(G)=\rad(G)-1$.

Kneser graphs are generalized Johnson graphs with $i=0$. Hence vertices of Kneser graphs $K(n,k)=J(n,k,0)$ are $k$-subsets of the $n$-set $X$ and two vertices $A,B \in V(K(n,r))$ are adjacent if and only if $A \cap B=\emptyset$. Some Kneser graphs are distance-transitive, for example odd graphs. However, this is not true for all Kneser graphs and thus generalized Johnson graphs are not distance-transitive in general. Next, we prove that the bound from Theorem~\ref{thm:boundRad} is also tight for all generalized Johnson graphs.

\begin{theorem}\label{thm:GeneralizedJohnson}
   If $G$ is a generalized Johnson graph, then $\rc(G)=\rad(G)-1$. 
\end{theorem}
\begin{proof}
    Let $n > k > i$ be non-negative integers. By Theorem~\ref{thm:boundRad}, $\rc(G) \leq \rad(J(n,k,i))-1$. 

    Now, let $A$ and $B$ be arbitrary vertices of $J(n,k,i)$ and let $|A\cap B|=\ell$ for some $\ell \in \{0,1, \ldots, k-1\}$. Let $A=\{a_1, \ldots, a_\ell, a_{\ell+1}, \ldots, a_k\}$ and $B=\{a_1, \ldots, a_\ell, b_{\ell+1}, \ldots, b_k\}$. Define $f:\{1, \ldots, n\} \to \{1, \ldots, n\}$,
      
    \[
    f(x) = 
      \begin{cases}
        b_i, & \text{if } x=a_i \text{ for some } i \in \{\ell +1,\ldots ,k\},\\
        a_i, & \text{if } x=b_i \text{ for some } i \in \{\ell +1,\ldots ,k\},\\
        x, & \text{ otherwise.}
      \end{cases}
    \]

    Moreover, let $F:V(J(n,k,i)) \to V(J(n,k,i))$, $F(X)=\{f(x)\mid x \in X\}$.
    Clearly $f$ is a bijection. Since $|X \cap Y|=|F(X) \cap F(Y)|$ holds for any $X,Y \in V(J(n,k,i))$, $F$ maps edges to edges and non-edges to non-edges and thus $F$ is an automorphism. Since $F(A)=B$ and $F(B)=A$, we have proved that for arbitrary $A,B \in V(J(n,k,i))$, there exists an automorphism $F$ that maps $A$ to $B$ and $B$ to $A$. Therefore $J(n,k,i)$ is a generously-transitive and thus Theorem~\ref{thm:strongAut} implies that $\rc(J(n,k,i))=\rad(J(n,k,i))-1$. 
    %
    %
\end{proof}

There also exist vertex-transitive graphs $G$ with $\rc(G) \leq \rad(G)-2$. To see this we need the following lemma.

\begin{lemma}\label{l:distance}
    Let $G$ be a connected graph and $i \leq \rad(G)$ an non-negative integer. If for any two vertices $x, y \in V(G)$ such that $d(x,y)=i$ there exists $y' \in N(y)$ with $d(x,y')=i+1$, then $\rc(G) \geq i$.  
\end{lemma}
\begin{proof}
    Let $i$ be a non-negative integer such that for any two vertices $x, y \in V(G)$ such that $d(x,y)=i$ there exists $y' \in N(y)$ with $d(x,y')=i+1$. We present a strategy for the robber to win the cop and robber game with radius of capture $i-1$. Regardless of the cop's first move, the robber selects the antipodal vertex of the vertex chosen by the cop. If in some step of the game, after the cop's move, the distance between both players is $i$, then robber in his next move selects its neighbor that is at distance $i+1$ from the cop, since such a vertex exists by assumption. Hence, the cop cannot capture the robber in the cop and robber game with radius of capture $i-1$ and therefore $\rc(G) \geq i$.
\end{proof}

Figure~\ref{fig:graph-cd-rad-2} shows a graph $G$ called \emph{CubicVT[24,6]} in the census of small connected cubic vertex-transitive graphs provided by Potočnik et.\ al.\ \cite{PoSp2013}. This graph has order 24 and radius 5. Moreover, for every vertex $v \in V(G)$ there exists exactly one vertex $u \in V(G)$ with $d(u, v)=5$. We first show that $G \in {\cal{CWRC}}(3)$. Let $c_1$ be the first cop's move. If the robber does not want to lose the game in the cop's next move, then the robber in his first move selects the vertex $r_1$ with $d(c_1, r_1)=5$. Next, in his second move the cop selects the vertex $c_2$ (see Figure~\ref{fig:graph-cd-rad-2}). Since $d(c_2, r) \leq 4$ holds for any $r \in N[r_1]$, the robber will, regardless of his next move, lose the cop and robber game with radius of capture 3 in the cop's next move. Thus, $G \in {\cal{CWRC}}(3)$ and consequently $\rc(G) \leq 3$. Since every vertex $x \in V(G)$ that is at distance 3 from $c_1$ has a neighbor $x'$ with $d(c_1, x')=4$, using that $G$ is vertex-transitive Lemma~\ref{l:distance} implies that $\rc(G) \geq 3$. Hence, the graph $G$ depicted in Figure~\ref{fig:graph-cd-rad-2} is an example of a vertex-transitive graph with $\rc(G)=\rad(G)-2$. Note, this graph is also a Cayley graph, thus not all Cayley graphs achieve the bound from Theorem~\ref{thm:boundRad}. 

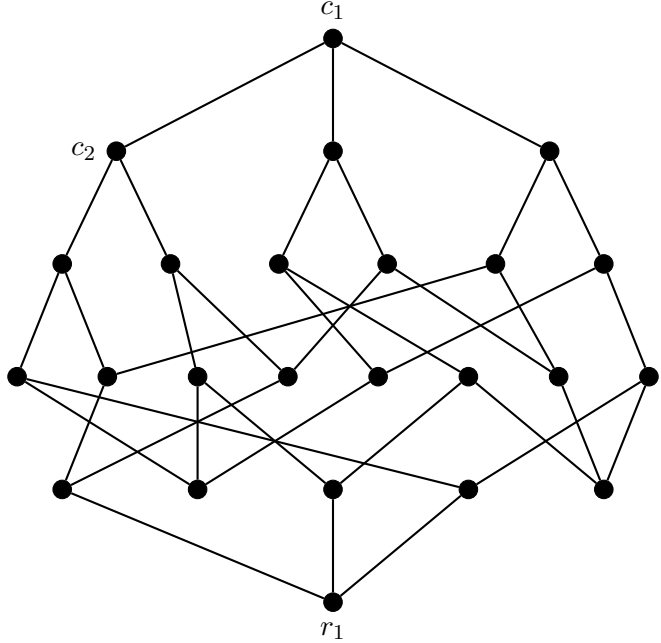
\begin{figure}[!ht]
    \centering
    \begin{tikzpicture}[scale=0.6, main_node/.style={circle, fill=black, draw, inner sep=0pt, minimum size=7pt]}]

\node[main_node] (0) at (0,12.5) [label=above:$c_1$]{};

\node[main_node] (1) at (-4.8, 10) [label=left:$c_2$]{};
\node[main_node] (2) at (4.8, 10) {};
\node[main_node] (3) at (0, 10) {};

\node[main_node] (18) at (3.6, 7.5) {};
\node[main_node] (19) at (1.2, 7.5) {};
\node[main_node] (20) at (-6, 7.5) {};
\node[main_node] (21) at (-3.6, 7.5) {};
\node[main_node] (22) at (6, 7.5) {};
\node[main_node] (23) at (-1.2, 7.5) {};

\node[main_node] (4) at (-6, 2.5) {};
\node[main_node] (6) at (6, 2.5) {};
\node[main_node] (7) at (-3, 2.5) {};
\node[main_node] (8) at (3, 2.5) {};
\node[main_node] (9) at (0, 2.5) {};

\node[main_node] (16) at (-1, 5) {};
\node[main_node] (10) at (-3, 5) {};
\node[main_node] (11) at (-7, 5) {};
\node[main_node] (17) at (-5, 5) {};
\node[main_node] (15) at (5, 5) {};
\node[main_node] (13) at (7, 5) {};
\node[main_node] (14) at (1, 5) {};
\node[main_node] (12) at (3, 5) {};

\node[main_node] (5) at (0, 0) [label=below:$r_1$]{};

 \path[draw, thick]
(0) edge node {} (1) 
(0) edge node {} (2) 
(0) edge node {} (3) 
(1) edge node {} (20) 
(1) edge node {} (21) 
(2) edge node {} (18) 
(2) edge node {} (22) 
(3) edge node {} (19) 
(3) edge node {} (23) 
(4) edge node {} (5) 
(4) edge node {} (16) 
(4) edge node {} (17) 
(5) edge node {} (8) 
(5) edge node {} (9) 
(6) edge node {} (12) 
(6) edge node {} (13) 
(6) edge node {} (15) 
(7) edge node {} (10) 
(7) edge node {} (11) 
(7) edge node {} (14) 
(8) edge node {} (11) 
(8) edge node {} (13) 
(9) edge node {} (10) 
(9) edge node {} (12) 
(10) edge node {} (21) 
(11) edge node {} (20) 
(12) edge node {} (23) 
(13) edge node {} (22) 
(14) edge node {} (22) 
(14) edge node {} (23) 
(15) edge node {} (18) 
(15) edge node {} (19) 
(16) edge node {} (19) 
(16) edge node {} (21) 
(17) edge node {} (18) 
(17) edge node {} (20) 
;

\end{tikzpicture}
    \caption{The graph \emph{CubicVT[24,6]} from the census of small connected cubic vertex-transitive graphs \cite{PoSp2013}.}
    \label{fig:graph-cd-rad-2}
\end{figure}

We have computed the radius capture number of all graphs available in the census of small connected cubic vertex-transitive graphs \cite{PoSp2013} and vertex-transitive graphs $G$ for which the difference $\rad(G)-\rc(G)$ is greater than 2 also exist. For example, the graph \emph{CubicVTgraph[56,12]} from the census was the highest order graph for which the radius capture number is equal to the half of the radius, namley $\rad(\text{\emph{CubicVTgraph[56,12]}})=6$ and $\rc(\text{\emph{CubicVTgraph[56,12]}})=3$. On the other hand, no tested graph had the radius capture number smaller than one half of the radius, moreover we found seven graphs from the census that achieve this lower bound. Hence, we propose the following question. 


\begin{question}
    If $G$ is a vertex-transitive graph, can the lower bound from Theorem \ref{thm:BoundGirth} be improved? Can it be proven that $\rc(G) \geq \frac{\rad(G)}{2}$?
\end{question}

We finish the section with the following question which is a special case of Problem~\ref{prob:1}.

\begin{question}
    Which vertex-transitive graphs $G$ have the radius capture number  equal to $\rad(G)-1$?
\end{question}

\section{Other graph classes}\label{sec:other}

\subsection{Outerplanar graphs}
\label{sec:outerplanar}

An {\em outerplanar} graph is a planar graph that can be embedded in the plane so that all its vertices lie on the same face, called the outer face. 

Note that outerplanar graphs $G$  with $\rc(G) \leq 1$ were characterized in~\cite{Dang-2011} as outerplanar graphs having all inner faces of length at most 5. In the next theorem we generalize this result and present explicit formula for $\rc(G)$, where $G$ is an arbitrary outerplanar graph. 

\begin{theorem}\label{thm:outerplanar}
    Let $G$ be an outerplanar graph embedded in the plane and $C$ a longest inner face of $G$. Then $\rc(G)= \left\lfloor \frac{|V(C)|}{2}\right\rfloor-1$.
\end{theorem}
\begin{proof}
    We first present a robber's strategy which guaranties that at every step of the game the distance between the cop and the robber is at least $\left\lfloor \frac{|V(C)|}{2}\right\rfloor-1$. Define a mapping $\sigma: V(G) \to 2^{V(C)}$ such that $\sigma$ maps $v \in V(G)$ to the set $\{x \in V(C) \mid d(v,C)=d(v,x)\}$. We call $\sigma(v)$ the \emph{set of shadows} of $v$. The cop will play the game optimally on $G$ and the robber will imagine the cop's moves on $C$ and respond by only playing on $C$. If the cop moves to $c_i \in V(G)$, then the robber imagines that the cop is on a vertex $c_i' \in \sigma(c_i)$ on $C$ and the robber's strategy is to move to an antipodal vertex $r_i \in V(C)$ of $c_i'$ on $C$. Since $C \subseteq G$, the cop continues playing on $G$ by responding to the robber's move $r_i$. Note, that by using this strategy, after each robber's move the distance between the cop and  the robber is at least $\left\lfloor \frac{|V(C)|}{2}\right\rfloor$. 

    We use induction on the number of moves to show that the robber can always move as described by the above strategy. First, since $G$ is outerplanar we note that for any $v \in V(G)$, it holds true that $|\sigma(v)| \leq 2$ and $G[\sigma(v)]$ is either isomorphic to $K_1$ or $K_2$. Hence, the set of antipodal vertices of all vertices in $\sigma(v)$ is of cardinality at most 2, if $C$ is of even length, and of cardinality at most 3, if $C$ is of odd length. Denote by $A(\sigma(v))$ the set of antipodal vertices (with respect to $C$) of all vertices in $\sigma(v)$. Since $G$ is outerplanar we can easily deduce the following remark.
    \begin{remark}\label{claim1}
        If $vu \in E(G)$, then for any $x \in A(\sigma(v))$ there exists $y \in A(\sigma(u))$ such that $y \in N[x]$.
    \end{remark}
    
    If $c_1$ is the vertex occupied by the cop in his first move and the robber responds by choosing an antipodal vertex of a vertex from $\sigma(c_1)$, say $r_1$, then $d(c_1,r_1) \geq \left\lfloor \frac{|V(C)|}{2}\right\rfloor$. Let $c_i$ be the vertex occupied by the cop in his $i^{\text{th}}$ move and let $r_i$ be an antipodal vertex of a shadow of $c_i$ that is occupied by the robber in his $i^{\text{th}}$ move (note that such vertex exists by the induction hypothesis). Let $c_{i+1}$ be the neighbor of $c_i$ occupied by the cop in his next move. Since $r_i \in A(\sigma(c_i))$ and $c_ic_{i+1} \in E(G)$, by Remark~\ref{claim1} there exists $r_{i+1} \in A(\sigma(c_{i+1}))$ such that $r_{i+1}\in N[r_i]$. Hence, if the robber moves to $r_{i+1}$, he has a strategy to be at distance $\left\lfloor \frac{|V(C)|}{2}\right\rfloor$ from the cop after each of his moves and thus to stay at distance $\left\lfloor \frac{|V(C)|}{2}\right\rfloor-1$ from the cop at each time of the game. Thus $\rc(G) \geq \left\lfloor \frac{|V(C)|}{2}\right\rfloor-1$.

    For the converse, let $G^*$ be the dual of $G$ and let $G'=G^*-\{F\}$, where $F$ is a vertex of $G^*$ that represents the outer face of $G$. Then $G'$ is a tree~\cite{Syslo-1979}. We will present a strategy for the cop which will ensure that cop wins the cop and robber game with radius of capture $\left\lfloor \frac{|V(C)|}{2}\right\rfloor-1$. Let $c_1$ and $r_1$ be the vertices occupied by the cop and robber in their first moves, respectively. Let $F_1$ be a face containing $c_1$ and $F_1'$ a face containing $r_1$. Let $R_1'$ be the shortest $F_1,F_1'$-path in $G'$ and let $R_1$ be a shortest $c_1,r_1$-path in $G$. Then cop has the strategy to move from $c_1$ to its neighbor on $R_1.$ In a similar way, for any $i>1$, let $c_i$ and $r_i$ be the vertices occupied by the cop and the robber in their $i^{\text{th}}$ move, respectively. Let $F_i$ be a face containing $c_i$ and $F_i'$ a face containing $r_i$. Let $R_i'$ be the shortest $F_i,F_i'$-path in $G'$ and let $R_i$ be a shortest $c_i,r_i$-path in $G$. Then the cop's strategy is to move from $c_i$ to its neighbor on $R_i$. Note, this vertex is on the face $F'$ which lies on $R_i'$. Since $G$ is outerplanar there exists $i$ such that $c_i$ and $r_i$ are on the same inner face of $G$. Thus $d_G(c_i,r_i) \leq \left\lfloor \frac{|V(F'')|}{2}\right\rfloor \leq \left\lfloor \frac{|V(C)|}{2}\right\rfloor$. Hence, after the cop's next move the distance between the cop and the robber is at most $\left\lfloor \frac{|V(F'')|}{2}\right\rfloor-1$, which implies that $\rc(G) \leq \left\lfloor \frac{|V(C)|}{2}\right\rfloor -1$.  
\end{proof}

\subsection{Sierpiński graphs}
\label{sec:sierpinski}

For integers $n \geq 0$ and $k \geq 1$, the \emph{Sierpiński graph} $S(n,k)$ is the graph with vertex set $V(S(n,k)) = [k]^n$ and edge set defined recursively as $E(S(0,k)) = \emptyset$ and $$E(S(n,k)) = \{ \{ix, iy\} \mid i \in [k], \{x,y\} \in E(S(n-1,k)) \} \cup \{ \{ij^{n-1}, ji^{n-1}\} \mid i,j \in [k], i \neq j \}.$$ For example, $S(1,3) \cong K_3$ and $S(3,3)$ is presented in Figure \ref{fig:sierpinski}. For more on Sierpiński graphs see for example~\cite{HKZ-2017}, and for results about the cop number of Sierpiński graphs, see~\cite{CA-2024}.

\begin{figure}[!ht]
    \centering
    \begin{tikzpicture}[scale=0.8, main_node/.style={circle, fill=black, draw, inner sep=0pt, minimum size=5pt]}]

\node[main_node] (111) at (4.5, 7.8) [label=above:\small$111$]{};
\node[main_node] (222) at (0,0) [label=below:\small$222$]{};
\node[main_node] (333) at (9,0) [label=below:\small$333$]{};

\node[main_node] (211) at (1.5, 2.6) [label=left:\small$211$]{};
\node[main_node] (122) at (3, 5.2) [label=left:\small$122$]{};
\node[main_node] (311) at (7.5, 2.6) [label=right:\small$311$]{};
\node[main_node] (133) at (6, 5.2) [label=right:\small$133$]{};
\node[main_node] (233) at (3,0) [label=below:\small$233$]{};
\node[main_node] (322) at (6,0) [label=below:\small$322$]{};

\node[main_node] (223) at (1,0) [label=below:\small$223$]{};
\node[main_node] (232) at (2,0) [label=below:\small$232$]{};
\node[main_node] (323) at (7,0) [label=below:\small$323$]{};
\node[main_node] (332) at (8,0) [label=below:\small$332$]{};

\node[main_node] (221) at (0.5, 0.87) [label=left:\small$221$]{};
\node[main_node] (212) at (1, 1.73) [label=left:\small$212$]{};
\node[main_node] (231) at (2.5, 0.87) [label=right:\small$231$]{};
\node[main_node] (213) at (2, 1.73) [label=right:\small$213$]{};

\node[main_node] (121) at (3.5, 6.07) [label=left:\small$121$]{};
\node[main_node] (112) at (4, 6.93) [label=left:\small$112$]{};
\node[main_node] (131) at (5.5, 6.07) [label=right:\small$131$]{};
\node[main_node] (113) at (5, 6.93) [label=right:\small$113$]{};

\node[main_node] (123) at (4, 5.2) [label=below:\small$123$]{};
\node[main_node] (132) at (5, 5.2) [label=below:\small$132$]{};

\node[main_node] (331) at (8.5, 0.87) [label=right:\small$331$]{};
\node[main_node] (313) at (8, 1.73) [label=right:\small$313$]{};
\node[main_node] (321) at (6.5, 0.87) [label=left:\small$321$]{};
\node[main_node] (312) at (7, 1.73) [label=left:\small$312$]{};

\path[draw, thick]
(111) edge node {} (222) 
(222) edge node {} (333) 
(333) edge node {} (111) 
(122) edge node {} (133)
(211) edge node {} (233)
(322) edge node {} (311)
(221) edge node {} (223)
(212) edge node {} (213)
(232) edge node {} (231)
(121) edge node {} (123)
(112) edge node {} (113)
(132) edge node {} (131)
(321) edge node {} (323)
(312) edge node {} (313)
(332) edge node {} (331)
;

\end{tikzpicture}
    \caption{The Sierpiński graph $S(3,3)$.}
    \label{fig:sierpinski}
\end{figure}
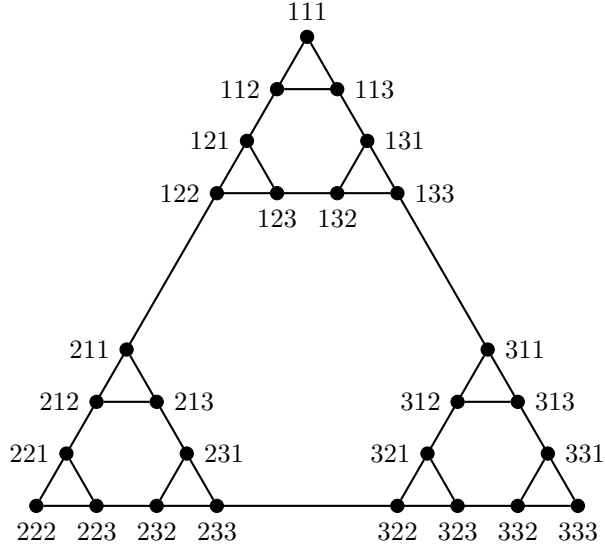

\begin{theorem}
    \label{thm:sierpinski-3}
    If $n \geq 3$ is an integer, then $\rc(S(n,k)) = 3 \cdot 2^{n-2} - 1$.
\end{theorem}

\begin{proof}
    It is known \cite{Pa2009} that $\rad(S(n,k)) = (2^{k-1} - 1) 2^{n-k+1}$ if $n \geq k$, thus $\rad(S(n,3)) = 3 \cdot 2^{n-2}$. Thus by Theorem \ref{thm:boundRad} we know that $\rc(S(n,k)) \leq 3 \cdot 2^{n-2} - 1$. To prove the equality we provide a strategy for the robber.

    Consider the following map $\sigma^1_n$ from $V(S(n,3))$ to $B = \{x \in V(S(n,3)) \mid 1 \notin x\}$. Let $\sigma^1_n(1^n) = 2 3^{n-1}$. For $11x \in V(S(n,3))$, if $11x$ is closer to $12^{n-1}$ than to $13^{n-1}$, then $\sigma^1_n(11x) = 2 3^{n-1}$, and otherwise, $\sigma^1_n(11x) = 32^{n-1}$. For $12x \in V(S(n,3))$, if $12x$ is at distance $d$ from $12^{n-1}$, then $\sigma^1_n(12x)$ is the vertex in $\{ 2y \in V(S(n,k)) \mid 1 \notin y \}$ that is at distance $2^{n-2} - 1 - d$ from $23^{n-1}$ (note that $\max \{ d(12x, 12^{n-1}) \mid x \in S(n-2,3) \} = 2^{n-2}-1$). For $13x \in V(S(n,3))$, if $13x$ is at distance $d$ from $13^{n-1}$, then $\sigma^1_n(13x)$ is the vertex in $\{ 3y \in V(S(n,k)) \mid 1 \notin y \}$ that is at distance $2^{n-2} - 1 - d$ from $32^{n-1}$. For $jx \in V(S(n,k))$, $j \neq 1$, $\sigma^1_n(jx) = j \sigma^1_{n-1}(x)$. We also define $\sigma^1_1(1) = \sigma^1_1(2) = 2$ and $\sigma^1_1(3) = 3$. Observe that if $xy \in E(S(n,3))$, then $\sigma^1_n(x) = \sigma^1_n(y)$ or $\sigma^1_n(x) \sigma^1_n(y) \in E(S(n,3)[B])$. Notice also that analogously we can define $\sigma^2_n \colon V(S(n,3)) \to \{x \in V(S(n,3)) \mid 2 \notin x\}$ and $\sigma^3_n \colon V(S(n,3)) \to \{x \in V(S(n,3)) \mid 3 \notin x\}$.

    The following vertices in $S(n,3)$ induce a cycle $C$: $\{1x \mid 1 \notin x\} \cup \{2x \mid 2 \notin x\} \cup \{3x \mid 3 \notin x \}$. Cycle $C$ is of length $3 \cdot 2^{n-1}$. Recall that $S(n,3)$ consists of three copies of $S(n-1,3)$ together with edges $\{12^{n-1}, 21^{n-1}\}$, $\{23^{n-1}, 32^{n-1}\}$ and $\{31^{n-1}, 13^{n-1}\}$ connecting them, and that $C$ contains all these three connecting edges and a path from each copy of $S(n-1,3)$. Define the map $\sigma \colon V(S(n,3)) \to C$ as $$\sigma(x_1 \ldots x_n) = x_1 \sigma^{x_1}_{n-1} (x_2 \ldots x_n).$$ It follows from the properties of $\sigma^i_n$ that if $xy \in E(S(n,3))$, then $\sigma(x) = \sigma(y)$ or $\sigma(x) \sigma(y) \in E(S(n,3)[C])$.

    As $C$ is a cycle of even length, every vertex on $C$ has a well-defined antipodal vertex. Robber's strategy when the game is played on $S(n,3)$ is the following: if cop is on vertex $c$, robber moves to (or stays on) the antipodal vertex of $\sigma(c)$ on $C$. The properties of $\sigma$ imply that this strategy is well-defined and that $d(c,r) \geq d(\sigma(c), \sigma(r))$ for every position $c$ of the cop and $r$ of the robber. Thus robber is always at least $\frac{|C|}{2} - 1 = 3 \cdot 2^{n-2} - 1$ away from the cop.
\end{proof}

Note that Theorem \ref{thm:sierpinski-3} essentially states that $\rc(S(n,3)) = \rad(S(n,3)) - 1$. It is not hard to see that this holds for $n \in \{0,1,2\}$ as well using that $\rad(S(n,k)) = 2^n - 1$ if $n < k$ (\cite{Pa2009}).

For $k\geq 4$ we computed (using computer) the radius capture number  for some small values of $n$. We noticed that $\rc(S(n,k))=\rad(S(n,k))-1$ does not hold in general. For example, $\rc(S(3,4)) = 5 \neq 6 = \rad(S(3,4))-1$ and $\rc(S(4,4)) = 11 \neq 13 = \rad(S(4,4))-1$. We wonder how the radius capture number  number behaves for $k \geq 4$ and thus finish the section by posing the following question.

\begin{question}
    What is $\rc(S(n,k))$ for $k \geq 4$?
\end{question}

\subsection{Harmonic even graphs}\label{sec:even}

A connected graph is called \emph{even} if, for any vertex $v \in V (G)$, there exists a unique vertex $v' \in V(G)$ such that $d(v,v') = \diam(G)$. Hence, for every even graph $G$ it holds true that $\rad(G)=\diam(G)$. An even graph is called \emph{harmonic even}, if $uv\in E(G)$ whenever $u'v' \in E(G)$ for all $u,v \in V(G)$, see \cite{GoVe86} for results on even and harmonic even graphs. We will also need the following result.

\begin{lemma}[\cite{GoVe86}]\label{lemma:evenD}
If $u$ and $v$ are adjacent vertices of an even graph $G$ of diameter $d$, then $d(u,v')=d-1$.
\end{lemma}

\begin{proposition}\label{prop:harmonic-even}
If $G$ is a harmonic even graph, then $\rc(G)=\rad(G)-1$.
\end{proposition}

\begin{proof}
    Let $G$ be a harmonic even graph. If the cop starts in the vertex $v\in V(G)$, then the robber can start in $v'$, hence $d(v',v)=\rad(G)$. When the cop moves to an adjacent vertex $u$, by Lemma \ref{lemma:evenD} $d(u,v')=\rad(G)-1$. Hence, $\rc(G)\leq \rad(G)-1$. Next, since $uv\in E(G)$ and $G$ is harmonic even, there exists $u'$ such that $v'u'\in E(G)$. The robber's next step is to move to $u'$ and increase the distance to the cop to $\rad(G)$. The robber can keep doing this strategy for any vertex to which the cop moves, therefore $\rc(G)\geq \rad(G)-1$. 
\end{proof}

Let $n$ be a positive integer. The hypercube of order $n$ is the graph $Q_n$, with the vertex set equal to all binary strings of length $n$ and in which two vertices are adjacent whenever the Hamming distance between the corresponding strings is exactly one. It is well-known that $\diam(Q_n) = n$. Also, hypercubes are harmonic even graphs \cite{GoVe86}. The following corollary is immediate consequence of this fact and Proposition \ref{prop:harmonic-even}.

\begin{corollary}
If $Q_n$ is the hypercube of order $n$, then $\rc(Q_n)=n-1$.
\end{corollary}


\subsection{Graph products} 

Let $G$ and $H$ be two graphs. Recall that for all of the standard graph products, the vertex set of a product of the graphs $G$ and $H$ is equal to $V(G)\times V(H)$. In the \emph{lexicographic product} $G\circ H$ of graphs $G$ and $H$, two vertices
$(g_{1},h_{1})$ and $(g_{2},h_{2})$ are adjacent whenever either
$g_{1}g_{2}\in E(G)$ or ($g_{1}=g_{2}$ and $h_{1}h_{2}\in E(H)$). In
the \emph{strong product} $G\boxtimes H$ of graphs $G$ and $H$
two vertices $(g_{1},h_{1})$ and $(g_{2},h_{2})$ are adjacent whenever
($g_1g_2\in E(G)$ and $h_1=h_2$) or ($g_1=g_2$ and $h_1h_2\in E(H)$)
or ($g_1g_2\in E(G)$ and $h_1h_2\in E(H)$). Finally, in the
\emph{Cartesian product} $G\BoxProduct H$ of graphs $G$ and $H$ two
vertices $(g_{1},h_{1})$ and $(g_{2},h_{2})$ are adjacent when
($g_1g_2\in E(G)$ and $h_1=h_2$) or ($g_1=g_2$ and $h_1h_2\in
E(H)$). By definition, we have $E(G\BoxProduct H)\subseteq E(G\boxtimes
H)\subseteq E(G\circ H)$.

Let $G$ and $H$ be graphs and $*\in\{\circ, \boxtimes, \BoxProduct\}$. For a vertex
$h\in V(H)$, we call the set $G^{h}=\{(g,h)\in V(G * H) \mid g\in
V(G)\}$ a $G$-\emph{layer} of $G * H$. By abuse of notation we will
also consider $G^{h}$ as the corresponding induced subgraph. Clearly
$G^{h}$ is isomorphic to $G$. For $g\in V(G)$, the $H$-\emph{layer}
$^g\!H$ is defined as $^g\!H =\{(g,h)\in V(G * H) \mid h\in V(H)\}$. We
will again also consider $^g\!H$ as an induced subgraph and note
that it is isomorphic to $H$. The map $p_{G}:V(G * H)\rightarrow
V(G)$, $p_{G}((g,h)) = g$ is the \emph{projection onto $G$}, and
$p_{H}:V(G * H)\rightarrow V(H)$, $p_{H}((g,h)) = h$ the
\emph{projection onto $H$}. We say that $G * H$ is \emph{nontrivial}
if both factors are graphs on at least two vertices.

Next, we present lower and upper found for the radius capture number  of Cartesian product of two graphs. For this we need the following lemma.

\begin{lemma}\label{l:farEnough}
    If $G$ is a connected graph with $\rc(G)=k$, then there exist sequences $(c_i)_{i\in \mathbb{N}}$ and $(r_i)_{i\in \mathbb{N}}$ of vertices in $G$, such that $c_ic_{i+1}\in E(G)$, $r_ir_{i+1} \in E(G)$, $d(c_i, r_i) \geq k+1$ and $d(c_{i+1},r_i) \geq k$ for any integer $i$.
\end{lemma}
\begin{proof}
    Let $c_1, \ldots, c_{\ell}$ and $r_1, \ldots, r_{\ell-1}$ be optimal moves of the cop and robber, respectively, when they play the cop and robber game with radius of capture $k$, such that the game ends after $\ell$-th move of the cop and $\ell$ is the smallest possible. Since $\rc(G)=k$ the game ends when the distance between both players is at most $k$. Thus $d(c_\ell,r_{\ell-1})=k$ (note that if both players play optimally, then the cop will have the last move) and $d(c_i,r_i) > k$ for any $i \in \{1, \ldots, \ell-1\}$ and $d(c_{i+1}, r_i) > k$ for any $i \in \{1, \ldots, \ell-2\}$. 
    We need to prove that they can continue moving and stay at distance at least $k$ after each of the cop's moves and at distance at least $k+1$ after each of the robber's moves. Suppose that both players continue moving with accordance to the rules of the game. Clearly, if $d(c_i, r_i) \geq k+1$, then $d(c_{i+1}, r_i) \geq k$, as $c_{i+1}\in N[c_i]$. Suppose that there exists integer $j$ such that $d(c_{j+1}, r_j) \geq k$ and $r_j$ has no neighbor that is at distance at least $k+1$ from $c_{j+1}$. Hence, $d(c_{j+1}, r_{j+1}) \leq k$ and thus cop can for his next move select a neighbor $c_{j+2}$ of $c_{j+1}$ from a shortest $c_{j+1}, r_{j+1}$-path. Hence $d(c_{j+2},r_{j+1}) < k$, which means that the cop has a strategy to win the cop and robber game with radius of capture $k-1$, a contradiction with the fact that $\rc(G)=k$.  
\end{proof}

\begin{theorem}\label{thm:Cartesian}
    Let $G$ and $H$ be connected graphs. Then 
    \begin{enumerate}
        \item[(i)] $\rc(G\BoxProduct H)\leq \min{\{\rad{(G)}+\rc(H),\rad{(H)}+\rc(G)\}},$
        \item[(ii)] $\rc(G \BoxProduct H) \geq \rc(G)+\rc(H)+1.$
    \end{enumerate}
    
\end{theorem}

\begin{proof}
    (i) W.l.o.g. assume that $\min{\{\rad{(G)}+\rc(H),\rad{(H)}+\rc(G)\}}=\rad(G)+\rc(H)$. Let $k=\rad(G)$ and $\ell=\rc(H)$. Let $g \in V(G)$ be such that $\ecc(g)=\rad(G)$ and let $h_1$ be the cop's optimal first move in the cop and robber game with radius of capture $\ell$ played on $H$. We present a strategy for the cop to win the cop and robber game with radius of capture $k +\ell$ played on $G\BoxProduct H$. Let $c_1=(g, h_1)$ be the vertex chosen by the cop in his first move, and let $r_1=(g_1', h_1')$ be the robber's first move. The cop's strategy is to stay in $^g\!H$ and follow his strategy as if playing on $H$. Namely, if $c_i=(g, h_i)$ and the robber in his $i^{\text{th}}$ move moves from $r_{i-1}=(g_{i-1}',h_{i-1}')$ to $r_i=(g_i',h_i')$, then the cop imagines the robber's move from $h_{i-1}'$ to $h_i'$ on $H$ and responds by selecting $c_{i+1}=(g,h_{i+1})$, where $h_{i+1}$ is his optimal move on $H$. Since $\rc(H)=\ell$, there exists $i$ such that $d_H(h_{i+1},h_{i}')\leq \ell$. Since $\ecc(g)=\rad(G)$, we get $d_{G \BoxProduct H}((g,h_{i+1}),(g_i',h_i'))\leq k+\ell$, which completes the proof. 

    (ii) Let $k=\rc(G)$ and $\ell=\rc(H)$. We will present a strategy for the robber which guarantees him to keep distance at least $k+\ell+1$ from the cop at any time of the game. Let $c_1=(g_1,h_1)$ be the vertex selected by the cop in his first move. The robber then imagines two games, one played on $G$ and one played on $H$. Let $u_1$ be the robber's optimal response to the cops move $g_1$ on $G$ and $v_1$ the robber's optimal response to the cops move $h_1$ on $H$. In $G \BoxProduct H$ the robber's first move is to select $(u_1, v_1)$. Lemma~\ref{l:farEnough} implies that $d_G(g_1, u_1) \geq k+1$ and $d_H(h_1, v_1) \geq \ell+1$. Let $c_i=(g_i, h_i)$ be the vertex selected by the cop in his $i^{\text{th}}$ move. If this was the move in an $H$-layer, i.e.\ $g_i=g_{i-1}$, then the robber imagines the game on $H$ and selects his optimal move $v_i$ in $H$, i.e.\ $r_i=(u_{i-1},v_i)=(u_i, v_i)$ (Lemma~\ref{l:farEnough} implies that $d_H(h_i, v_{i-1}) \geq \ell$ and $d_H(h_i, v_i) \geq \ell+1$). Similarly, if the cop makes a move in a $G$-layer, i.e.\ $h_i=h_{i-1}$, then the robber imagines the game on $G$ and selects his optimal move $u_i$ in $G$, i.e.\ $r_i=(u_{i}, v_{i-1})=(u_i, v_i)$ (Lemma~\ref{l:farEnough} implies that $d_G(g_i, u_{i-1}) \geq k$ and $d_G(g_i, u_i) \geq k+1$). Using this strategy we get that $d_{G\BoxProduct H}((g_i,h_i),(u_i,v_i)) \geq k+l+2$ and $d_{G\BoxProduct H}((g_{i+1},h_{i+1}),(u_i,v_i)) \geq k+l+1$, which implies that $\rc(G \BoxProduct H) \geq k+\ell+1$.
\end{proof}

If at least one of the graphs $G$ or $H$ achieves the bound in Theorem~\ref{thm:boundRad}, say $G$ (i.e.\ $\rc(G)=\rad{(G)}-1$), then the upper and lower bounds from Theorem~\ref{thm:Cartesian} coincide, thus in this case $\rc(G\BoxProduct H)=\rc(G)+\rc(H)+1$. Moreover, this directly implies the radius capture number  of a Hamming graph $H(d,q)$, which is the Cartesian product of $d$ complete graphs $K_q$~\cite{BH-2011}. A special case of Hamming graphs are hypercubes $Q_d=H(d,2)$. Since the Cartesian product is associative, it holds true that for any $d \geq 2$, $H(d,q)=K_q \BoxProduct H(d-1,q)$. Using $\rc(K_q)=\rad{(K_q)}-1=0$, induction directly implies the following.

\begin{corollary}
    If $d,q$ are arbitrary positive integers, then $\rc(H(d, q))=d-1$.
\end{corollary}

Note that this result was obtained already in Section~\ref{sec:vertexTransitive}. Moreover, if $q=2$, the result was obtained also in Section~\ref{sec:even}.

In the next result we compote the radius capture number  of the strong product of two graphs.

\begin{theorem}
    Let $G$ and $H$ be connected graphs. Then $$\rc(G \boxtimes H)=\max{\{\rc(G),\rc(H)\}}.$$
\end{theorem}
\begin{proof}
    Suppose that the cop is positioned in $(g,h)$. To determine his next move, the cop considers the projections to $G$ and $H$ of his and the robber's positions, determines his optimal next move in each factor, say $g' \in G$ and $h' \in H$, respectively, and moves to $(g',h') \in V(G \boxtimes H)$. Clearly, $(g,h) (g',h') \in E(G)$. As he is playing optimally on $G$ and $H$, he can get to distance $\rc(G)$ to the robber's projection to $G$ and to distance $\rc(H)$ to the robber's projection on $H$. The proof of Lemma~\ref{l:farEnough} implies that when cop reaches to distance $\rc(G)$ on the projection to $G$, after each cop's move the distance between both players is again $\rc(G)$ and the same holds for the projections to $H$. More precisely, let $c_i=(g_i,h_i)$, $r_i=(u_i,v_i)$ (where $c_i, r_i$ denote the vertices selected by the cop and the robber in step $i$, respectively) and let $i$ be the smallest index with $d_G(g_i,u_{i-1}) = \rc(G)$ and $d_G(g_j,r_{j-1}) \geq \rc(G)+1$ for any $j < i$. Then Lemma~\ref{l:farEnough} implies that the cop has a strategy on $G$ to stay at distance $k$ from the robber after each of his moves. Thus $d_G(g_j,u_{j-1})=k$ holds for any $j \geq i$. Now let $i' \geq i$ be the smallest index with $d_H(h_{i'},v_{i'-1})=\rc(H)$. Hence $$d_{G \boxtimes H}((g_{i'},h_{i'}),(u_{i'-1},v_{i'-1}))=\max{\{d_G(g_{i'},u_{i'-1}),d_H(h_{i'},v_{i'-1})\}}=\max{\{\rc(G),\rc(H)\}},$$
    and thus the cop has a strategy to capture the robber at distance $\max{\{\rc(G),\rc(H)\}}$.

    Similarly, the robber's strategy is to consider projections to $G$ and $H$ and play optimally in the factors. Thus, he can maintain the distance at least $\rc(G)$ from the cop's projection on $G$ and at least $\rc(H)$ from the cop's projection on $H$. Hence, the robber is always at least $\max{\{\rc(G),\rc(H)\}}$ away from the cop on $G \boxtimes H$.
\end{proof}

Finally, we present exact value of $\rc(G \circ H)$ that depends mostly on $\rc(G)$.

\begin{theorem}
    Let $G$ and $H$ be non-trivial connected graphs. Then  
  \[\rc(G \circ H) = 
  \begin{cases}
      \rc(G), & \text{if } \rc(G) \geq 1,\\
    \min{\{1,\rc(H)\}}, & \text{if } \rc(G)=0.
  \end{cases}
\]
\end{theorem}

\begin{proof}
    First let $\rc(G) \geq 1$.  Suppose that the cop is positioned in $(g,h)$. If the robber selects the vertex $(g,h')$, then the cop can move to $(g',h)$ for arbitrary neighbor $g'$ of $g$ in $G$ and reach distance $1 \leq \rc(G)$ to the robber. Otherwise, the cop has the following strategy: he considers the projection of his and the robber's positions to $G$, determines his optimal next move on $G$, say $g' \in V(G)$, and moves to $(g',h') \in V(G \circ H)$, where $h'$ is an arbitrary vertex of $H$. Using this strategy he will eventually reach the distance at most $\rc(G)$ to the robber. 
    
    Conversely, the robber can consider the projection of his and the cop's positions to $G$, determine his optimal next move in $G$, say $u' \in V(G)$, and move to $(u',v') \in V(G \circ H)$, where $v'$ is an arbitrary vertex of $H$. Using this strategy in $G$, the robber can stay at distance at least $\rc(G)$ from the cop, consequently he is at the same distance also in $G \circ H$. Thus, $\rc(G \circ H)=\rc(G)$ when $\rc(G) \geq 1$.

    Now, let $\rc(G)=0$ and $\rc(H)=0$. Hence, both $G$ and $H$ are cop-win graphs and thus dismantlable graphs. Therefore $G \circ H$ is dismantlable (note that if $u$ is a corner in $G$ and $v$ a corner in $H$, then $(u,v)$ is a corner in $G \circ H$) and thus $\rc(G\circ H)=0$. 

    Finally, let $\rc(G)=0$ and $\rc(H) \geq 1$. Suppose that the cop is positioned in $(g,h)$. If the robber selects the vertex $(g,h')$, then the cop can move to $(g',h)$ for an arbitrary neighbor $g'$ of $g$ in $G$ and reach the distance $1$ to the robber. Otherwise (if the robber's and the cop's first moves are in different $H$-layers), the cop has the following strategy: he considers the projection of his and robber's positions to $G$, determines his optimal next move in $G$, say $g' \in V(G)$, and moves to $(g',h') \in V(G \circ H)$, where $h'$ is the projection to $H$ of the robber's current position. Using this strategy, the cop will eventually capture the robber since $\rc(G)=0$. Thus, the cop has a strategy to reach the distance at most 1 from the robber, which implies that $\rc(G \circ H) \leq 1$.
    
    Conversely, the robber has the following strategy. If in an arbitrary move cop selects the vertex $(g,h)$, then robber selects the vertex $(g,h')$ in the same $H$-layer, where $h'$ is an optimal robber's respond on $H$ to the cop's move $h$ on $H$, if cop played inside the $H$-layer, and otherwise (if cop moved to a neighboring layer) $h'$ is an optimal first move of the robber on $H$ if cop selects $h$ on $H$ in his first move. Using this strategy, after each of the robber's moves the distance between both players is at least 2 and after each of the cop's moves the distance between them is at least 1. Hence, the robber has a strategy to keep distance at least 1 in every stage of the game, which implies that $\rc(G \circ H) \geq 1$.
\end{proof}

\section{Concluding remarks}

As already mentioned, bipartite graphs that are in ${\cal{CWRC}}(1)$ have been characterized in~\cite{CCNV-2011}. We wonder if it is possible to characterize bipratite graphs that are in ${\cal{CWRC}}(k)$ for $k \geq 2$.

Let us consider the following situation. Two player are moving on the vertices of a graph, such that only one is allowed to move at a time and they both want visit all of the graph's vertices. For two such walks (one for each player) at each point in time we measure the distance between the players, the minimum of these distances is a safety distance that the players were able to maintain at all points in time during those two walks. Taking into consideration all possible pairs of such walks, we want to maximize the obtained safety distance. So the question is, what is the largest possible safety distance, two players are able to maintain at all points in time during traversal of the vertices of a given graph by obeying the movement rules and both visiting all vertices? The described invariant has been introduced under the name of \emph{the Cartesian vertex span} of the given graph $G$, denoted by $\cvSpan{G}$. For a more precise definition and similar invariants depending on other movement results, see~\cite{BT-2023}.

As this problem resembles the cop and robber game with radius of capture $k$, although here only the robber is trying to maximize the distance between the players and we can pose the restriction of visiting all vertices only to the cop, we wonder if the two graph invariants are related.

\begin{question}
    Is the Cartesian vertex span of a connected graph $G$ an upper bound for the radius capture number, i.e.  $\rc(G) \leq \cvSpan{G}$?
\end{question}

\section*{Acknowledgements}
The authors acknowledge the financial support from the Slovenian Research and Innovation Agency (research core funding No.\ P1-0297 and projects N1-0285, Z1-50003, N1-0218, and N1-0355). Vesna Iršič Chenoweth was also supported by the European Union (ERC, KARST, 101071836).



\end{document}